\documentclass[reqno]{amsart}

\usepackage{amssymb,amscd,amsopn,cite,mathrsfs}
\usepackage[centertags]{amsmath}
\usepackage{amsfonts,amsthm}
\usepackage{enumerate}
\usepackage{latexsym}
\usepackage{newlfont}
\usepackage{graphicx}
\usepackage{epsfig}
\usepackage{color}

\definecolor{darkred}{rgb}{0.8,0,0}

\definecolor{OliveGreen}{cmyk}{0.64,0,0.95,0.40}

\newcommand{\R}{\hbox{\ensuremath{\mathbb{R}}}}

\theoremstyle{plain}
\newtheorem{theorem}{Theorem}
\newtheorem{proposition}{Proposition}
\newtheorem{lemma}{Lemma}
\newtheorem{corollary}{Corollary}
\theoremstyle{definition}
\newtheorem{definition}{Definition}
\theoremstyle{remark}
\newtheorem{remark}{Remark}

\def\Hisp{{\boldsymbol{\mathcal H}}}
\def\Id{\operatorname{Id}}

\begin{document}

\title{\textsc{Critical pairs of sequences of a Mixed Frame Potential}}

\author{Ivana Carrizo$^*$ and Sigrid Heineken$^{\dag}$}

\thanks{$^*$  NuHAG, Department of Mathematics, University of Vienna, Nordbergstrasse 15, A-1090 Vienna, Austria. E-mail: ivana.carrizo@univie.ac.at}
\thanks{$^{\dag}$ Departamento de Matem\'atica, Facultad de Ciencias Exactas y Naturales, Universidad de Buenos Aires, Pabell\'on I, Ciudad Universitaria, C1428EGA C.A.B.A., and IMAS, CONICET, Argentina. E-mail: sheinek@dm.uba.ar.\newline
{\em Correspondence to}: Sigrid Heineken, Departamento de
Matem\'atica, Facultad de Ciencias Exactas y Naturales,
Universidad de Buenos Aires, Pabell\'on I, Ciudad Universitaria,
C1428EGA C.A.B.A., Argentina, tel/fax:+541145763335, E-mail:
sheinek@dm.uba.ar, sigrid.heineken@gmail.com.}

\begin{abstract}

The classical frame potential in a finite dimensional Hilbert space has been introduced by Benedetto and Fickus, who showed that all finite unit-norm tight frames can be characterized as the minimizers of this energy functional. This was the start point of a series of new results in frame theory,  related to finding tight frames with determined length. The frame potential has been studied in the traditional setting  as well as in the finite-dimensional fusion frame context.
In this work we introduce the concept of {\sl mixed frame potential}, which generalizes the
notion of the Benedetto-Fickus frame potential. We study properties of this new potential, and give the structure of its critical pairs of sequences on a suitable restricted domain.
For a given sequence $\{\alpha_m\} _{m=1,...,N}$ in $K,$ where $K$ is $\mathbb{R}$ or $\mathbb{C},$  we obtain necessary and sufficient conditions in order to have a dual pair of frames $\{f_m\}_{m=1,...,N}$, $\{g_m\}_{m=1,...,N}$ such that $\langle f_m, g_m \rangle=\alpha_m$ for all $m=1,..., N.$

\medskip

{\bf Key words:} Finite frames, frame potential, dual frames, Lagrange multipliers.

\medskip

{\bf AMS subject classification:} Primary: 42C15, 42C99, 42C40.

\end{abstract}


\maketitle

\pagestyle{myheadings}
\markboth{IVANA CARRIZO AND SIGRID HEINEKEN}{CRITICAL PAIRS OF SEQUENCES OF A MIXED FRAME POTENTIAL}

\pagenumbering{arabic}

\section{Introduction}
Frames, which were introduced by Duffin and Schaeffer in \cite{duffschaef1}, became essential for engineering and applied mathematics, specially for the purpose of signal processing and data transmission. Given a Hilbert space $\mathbb{H},$ a sequence $\{f_m\}\subset \mathbb{H}$ is a {\sl frame} if there exist positive constants $A$ and $B$ that satisfy
$$A\|f\|^2\leq\sum_m|\langle f,f_m\rangle|^2 \leq
B\|f\|^2\,\,\,\,\,\forall f \in \mathbb{H}.$$
If $A=B$ it is called a {\sl tight frame}.

The main property of frames is that they provide reconstruction formulae where the coefficients are not necessarily unique, which is advantageous in situations that arise in signal processing \cite{BC95}. Particular frames such as wavelet and Gabor frames are described e.g. in \cite{HW89}, \cite{D90}, \cite{D92}, \cite{ch03}.

Finite frames are used in many applications, where we often have to
work in finite dimensional spaces, since they avoid the
approximation problems that come up by truncating infinite frames.
They have been studied for example in \cite{befi03}, \cite{ca04},
\cite{GKK01}. In particular, finite tight frames are very useful to
solve problems in Communication Theory, Information Theory, Sampling
Theory, etc. \cite{SH03}, since the convergence of the provided decompositions is fast.
The {\sl frame  potential} in $\mathbb{H}_d^N$ - where
$\mathbb{H}_d$ is a finite dimensional Hilbert space - introduced in
\cite{befi03} by Benedetto and Fickus- turned out to be an important
tool in frame theory. In our work we define a new concept of potential in
$\mathbb{H}_d^N\times \mathbb{H}_d^N.$  Whereas the Benedetto-Fickus
potential measures the orthogonality of a system of vectors, our
{\sl mixed frame potential} quantifies in some sense the
biorthogonality of two systems of vectors.

In \cite{ca04} and \cite{cafikoletr06}, the problem of finding tight frames with a
prescribed norm is analyzed, which is related to the minimization of
the Benedetto-Fickus frame potential. The Benedetto-Fickus frame potential has been also studied in the finite-dimensional
fusion frame setting \cite{CF09}, \cite{MRS10}.

Given a sequence $\{\alpha_m\} _{m=1,...,N},$  we study the mixed frame potential restricted to the pairs
$(\{f_m\}_{m=1}^N,\{g_m\}_{m=1}^N)$ such that $\langle f_m, g_m \rangle=\alpha_m,$ and describe the
critical pairs of sequences of this restricted potential. This turned out to be related to finding dual pairs of frames that satisfy $\langle f_m, g_m \rangle=\alpha_m.$

The paper is organized as follows. In the following section we give
definitions and preliminaries that we will use later. In section 3
we present some properties of the mixed frame potential. Section 4
is devoted to characterize the structure of the critical pairs of sequences of the
mixed frame potential, which leads to the result about necessary and
sufficient conditions for the existence of dual frames with prescribed scalar products.

\section{Notation and Preliminaries}

Let $K$ be $\mathbb{R}$ or $\mathbb{C}$ and $\mathbb{H}_d$ a $d$-dimensional Hilbert space over $K.$  Let $\{f_m\}_{m=1}^N$ and $\{g_m\}_{m=1}^N$ be sequences in $\mathbb{H}_d.$
The synthesis operator for $\{f_m\}_{m=1}^N$ is given by
 $$T:K^N\rightarrow\mathbb{H}_d,~~T(\{c_m\}_{m=1}^N)=\sum_{m=1}^N c_m f_m$$
and the analysis operator for $\{f_m\}_{m=1}^N$ by $$T^\star: \mathbb{H}_d\rightarrow K^N,~~ T^\star(f)=\{\langle f,f_m\rangle \}_{m=1}^N.$$
We will denote with $U$ and $U^\star$ the synthesis and respectively analysis operator of $\{g_m\}_{m=1}^N.$
We denominate $TU^\star$ and $UT^\star$ the \textit{mixed frame operators}:

For $f\in \mathbb{H}_d$ we have
\begin{equation}\label{dual}
TU^\star (f)= \sum_{m=1}^N\langle f,g_m\rangle  f_m,~~~~~~\hbox{and}~~~~~UT^\star(f)=\sum_{m=1}^N\langle f,f_m\rangle  g_m.
\end{equation}

Two sequences $\{f_m\}_{m=1}^N$ and $\{g_m\}_{m=1}^N$ are \textit{dual frames} if
\begin{equation}\label{dual_1}
f=\sum_{m=1}^N\langle f,g_m\rangle  f_m \,\,\,\forall f\in \mathbb{H}_d~~~~ \hbox{or}~~~~f=\sum_{m=1}^N\langle f,f_m\rangle  g_m\,\,\,\forall f\in \mathbb{H}_d.
\end{equation}
In terms of the operators $T$ and $U,$ \eqref{dual_1} means that $TU^\star=I$ or $UT^\star=I.$

\bigskip

\begin{definition}\label{def_frame_pot}
Let $\widetilde{FP}:\mathbb{H}_d^N\times\mathbb{H}_d^N\longrightarrow K,$
$$\widetilde{FP}(\{f_m\}_{m=1}^N,\{g_m\}_{m=1}^N)=\sum_{m=1}^N\sum_{n=1}^N \langle f_m,g_n\rangle \langle f_n,g_m\rangle.$$
We call $\widetilde{FP}$ the \textit{mixed frame potential} of $(\{f_m\}_{m=1}^N,\{g_m\}_{m=1}^N)\in\mathbb{H}_d^N\times\mathbb{H}_d^N.$
\end{definition}

Observe that for the case that $\{f_m\}_{m=1}^N=\{g_m\}_{m=1}^N,$ the mixed frame potential is equal to $$FP(\{f_m\}_{m=1}^N)=\sum_{m=1}^N\sum_{n=1}^N |\langle f_m,f_n\rangle|^2,$$
 which is the traditional Benedetto-Fickus frame potential of $\{f_m\}_{m=1}^N.$

Given a sequence $\{\alpha_m\}_{m=1}^N\subset K$ we define
$$\tilde{S}(\{\alpha_m\}_{m=1}^N)=\left\{(\{f_m\}_{m=1}^N,\{g_m\}_{m=1}^N)\in \mathbb{H}_d^N\times \mathbb{H}_d^N: \langle f_m,g_m\rangle =\alpha_m \,\,\forall m=1,...,N\right\}.$$

\section{Mixed Frame Potential}

We will see next that the mixed frame potential can also be written as the trace of the square of the corresponding mixed frame operator, i.e. it is the square of the Hilbert-Schmidt norm of the mixed frame operator.

\begin{lemma}\label{FramePotAutov}
For any pair $(\{f_m\}_{m=1}^N,\{g_m\}_{m=1}^N)\in \mathbb{H}_d^N\times\mathbb{H}_d^N$ with corresponding mixed frame operator $TU^\star,$
\begin{equation}
\begin{aligned}
\widetilde{FP}(\{f_m\}_{m=1}^N,\{g_m\}_{m=1}^N)&=Tr((TU^\star)^2)=\sum_{n=1}^d\lambda_n^2
\end{aligned}
\end{equation}
where $\{\lambda_n\}_{n=1}^d$ are the eigenvalues of $TU^\star.$
\end{lemma}
\begin{proof}
Let $\{e_n\}_{n=1}^d$ be an orthonormal basis of $\mathbb{H}_d.$
\begin{equation*}
\begin{aligned}
    \widetilde{FP}(\{f_m\}_{m=1}^N,\{g_m\}_{m=1}^N) & = \sum_{m=1}^N\sum_{n=1}^N \langle f_m,g_n\rangle \langle  f_n,g_m\rangle \\
 &  =\sum_{m=1}^N\sum_{n=1}^N\langle \sum_{l=1}^d \langle f_m,e_l\rangle e_l,g_n\rangle \langle f_n,g_m\rangle \\
 &= \sum_{m=1}^N \sum_{n=1}^N \sum_{l=1}^d  \langle f_m,e_l\rangle \langle e_l,g_n\rangle\langle f_n,g_m\rangle\\
 &= \sum_{l=1}^d \langle  \sum_{n=1}^N \langle e_l,g_n\rangle f_n , \sum_{m=1}^N\langle e_l,f_m\rangle g_m \rangle \\
 &= \sum_{l=1}^d \langle  TU^\star e_l, UT^\star e_l\rangle = \sum_{l=1}^d \langle (UT^\star)^\star TU^\star e_l, e_l\rangle \\
 &= \sum_{l=1}^d \langle (TU^\star)^2 e_l, e_l\rangle  = Tr((TU^\star)^2)
\end{aligned}
\end{equation*}
Let $\{\lambda_n\}_{n=1}^d$ denote the eigenvalues of $TU^\star,$ counting multiplicities. Since the eigenvalues of $(TU^\star)^2$ are $\{\lambda_n^2\}_{n=1}^d$ we have that
$$Tr((TU^\star)^2)=\sum_{n=1}^d\lambda_n^2.$$
\end{proof}

\begin{remark}
Observe that $$\widetilde{FP}(\{g_m\}_{m=1}^N,\{f_m\}_{m=1}^N)=Tr((U^\star T)^2)=\sum_{n=1}^d\overline{\lambda_n}^2.$$
\end{remark}

Note that the previous result allows to compute the mixed frame potential very easily for example for a pair $(\{f_m\}_{m=1}^N,\{g_m\}_{m=1}^N)$ such that $TU^*=A\Id$ with $A\in K.$  In this case, $\widetilde{FP}(\{f_m\}_{m=1}^N,\{g_m\}_{m=1}^N)=A^2d.$

Also, the previous representation of the mixed frame potential allows us to study in more detail some of its properties, as we will see in the following proposition.

\begin{proposition}\label{cotasfp}
Let $\{\alpha_m\}_{m=1}^N\subset K$ and $(\{f_m\}_{m=1}^N,\{g_m\}_{m=1}^N)\in \tilde{S}(\{\alpha_m\}_{m=1}^N).$
\begin{enumerate}
\item If all the eigenvalues of $TU^\star$ are real, then $\widetilde{FP}(\{f_m\}_{m=1}^N,\{g_m\}_{m=1}^N)$ and $\sum_{m=1}^N\alpha_m$ are real and $$\widetilde{FP}(\{f_m\}_{m=1}^N,\{g_m\}_{m=1}^N)\ge \frac{1}{d}\left(\sum_{m=1}^N\alpha_m\right)^2.$$

\item If all the eigenvalues of $TU^\star$ are imaginary, then $\widetilde{FP}(\{f_m\}_{m=1}^N,\{g_m\}_{m=1}^N)$ is real and $\sum_{m=1}^N\alpha_m$ is imaginary and
    $$\widetilde{FP}(\{f_m\}_{m=1}^N,\{g_m\}_{m=1}^N)\le \frac{1}{d}\left(\sum_{m=1}^N\alpha_m\right)^2$$

\item If $TU^*$ has only one eigenvalue, then $$\widetilde{FP}(\{f_m\}_{m=1}^N,\{g_m\}_{m=1}^N)=\frac{1}{d}\left(\sum_{m=1}^N\alpha_m\right)^2.$$ In particular, this happens if $TU^*=\frac{1}{d}\left(\sum_{m=1}^N\alpha_m\right) \Id.$
\end{enumerate}
\end{proposition}

\begin{proof}

By the preceding lemma we know that if $(\{f_m\}_{m=1}^N,\{g_m\}_{m=1}^N)\in \mathbb{H}_d^N\times\mathbb{H}_d^N,$
\begin{equation}\label{ecua_estrella}
\begin{aligned}
\widetilde{FP}(\{f_m\}_{m=1}^N,\{g_m\}_{m=1}^N)&=\sum_{n=1}^d\lambda_n^2\\
&=\sum_{n=1}^d \left((Re(\lambda_n))^2-(Im(\lambda_n))^2\right)+ 2i\;\sum_{n=1}^d Re(\lambda_n) Im(\lambda_n)
\end{aligned}
\end{equation}
where $\{\lambda_n\}_{n=1}^d$ are the eigenvalues of $TU^\star.$

Let $\{e_n\}_{n=1}^d$ be an orthonormal basis for $\mathbb{H}_d.$ If $(\{f_m\}_{m=1}^N,\{g_m\}_{m=1}^N)\in \tilde{S}(\{\alpha_m\}_{m=1}^N) $ the trace of the mixed frame operator satisfies
\begin{equation*}
\begin{aligned}
\sum_{n=1}^d\lambda_n &= Tr(TU^\star)=\sum_{n=1}^d \langle TU^\star e_n,e_n\rangle =\sum_{n=1}^d \langle \sum_{m=1}^N\langle e_n,g_m\rangle f_m,e_n\rangle \\
&= \sum_{n=1}^d \sum_{m=1}^N \langle e_n,g_m\rangle  \langle f_m,e_n\rangle = \sum_{m=1}^N \sum_{n=1}^d \langle e_n,g_m\rangle  \langle f_m,e_n\rangle \\
&= \sum_{m=1}^N \langle f_m, g_m\rangle =\sum_{m=1}^N\alpha_m .
\end{aligned}
\end{equation*}

So, in order to study possible extrema for the real or the imaginary part of $\widetilde{FP}:\tilde{S}(\{\alpha_m\}_{m=1}^N)\longrightarrow K,$ we will first consider the critical points of the functions
\begin{equation*}
  \mathcal{R}(\lambda_1,..., \lambda_d)=\mathcal{R}\left( Re(\lambda_1),...,Re(\lambda_d),Im(\lambda_1),...,Im(\lambda_d)\right)=\sum_{n=1}^d (Re(\lambda_n))^2-(Im(\lambda_n))^2
 \end{equation*}
and
 \begin{equation*}
  \mathcal{I}(\lambda_1,..., \lambda_d)=\mathcal{I}\left(Re(\lambda_1),...,Re(\lambda_d),Im(\lambda_1),...,Im(\lambda_d)\right)=2\sum_{n=1}^d Re(\lambda_n) Im(\lambda_n)
 \end{equation*}
restricted to the set $\Lambda\subset\mathbb{C}^d\simeq \R^{2d}$, where $(\lambda_1,..., \lambda_d)\in\Lambda$ if and only if
\begin{equation*}
\sum_{n=1}^d Re(\lambda_n)=Re\left(\sum_{m=1}^N\alpha_m\right)\,\,\,\text{ and } \,\,\,
\sum_{n=1}^d Im(\lambda_n)=Im\left(\sum_{m=1}^N\alpha_m\right).
\end{equation*}

Using Lagrange multipliers for this constrained problem, we obtain that if

$(\lambda_1,..., \lambda_d)$ is a critical point of $\mathcal{R}$ or $\mathcal{I}$ restricted to $\Lambda,$ then
$$\lambda_1=\lambda_2= ... =\lambda_d=\frac{1}{d}\sum_{m=1}^N \alpha_m.$$
Furthermore, in this case it can be seen that
\begin{enumerate}
\item[(i)]  if $Im(\lambda_1,..., \lambda_d)=0$ then $(\lambda_1,..., \lambda_d)$ is a minimum of $\mathcal{R}$ restricted to $\Lambda$ and $\mathcal{I}(\lambda_1,..., \lambda_d)=0,$
\item[(ii)] if $Re(\lambda_1,..., \lambda_d)=0$ then $(\lambda_1,..., \lambda_d)$ is a maximum of $\mathcal{R}$ restricted to  $\Lambda$ and $\mathcal{I}(\lambda_1,..., \lambda_d)=0,$
\item[(iii)]if $Re(\lambda_1,..., \lambda_d)\neq 0$ and $Im(\lambda_1,..., \lambda_d)\neq 0,$ then $(\lambda_1,..., \lambda_d)$ is a saddle point of $\mathcal{R}$ as well as of $\mathcal{I}$ restricted to $\Lambda.$
\end{enumerate}

Thus, for any $(\{f_m\}_{m=1}^N,\{g_m\}_{m=1}^N)\in \tilde{S}(\{\alpha_m\}_{m=1}^N)$ such that all the eigenvalues of $TU^\star$ are real, we have that $\widetilde{FP}(\{f_m\}_{m=1}^N,\{g_m\}_{m=1}^N)$ and $\sum_{m=1}^N\alpha_m$ are real and
$$\widetilde{FP}(\{f_m\}_{m=1}^N,\{g_m\}_{m=1}^N)=\sum_{n=1}^d\lambda_n^2\ge \frac{1}{d}\left(\sum_{m=1}^N\alpha_m)\right)^2,$$
and for any $(\{f_m\}_{m=1}^N,\{g_m\}_{m=1}^N)\in \tilde{S}(\{\alpha_m\}_{m=1}^N)$ such that all the eigenvalues of $TU^\star$ are imaginary,
$\widetilde{FP}(\{f_m\}_{m=1}^N,\{g_m\}_{m=1}^N)$ is real and $\sum_{m=1}^N\alpha_m$ is imaginary and
 $$Im\left(\widetilde{FP}(\{f_m\}_{m=1}^N,\{g_m\}_{m=1}^N)\right)=\sum_{n=1}^d\lambda_n^2\le \frac{1}{d}\left(\sum_{m=1}^N\alpha_m\right)^2.$$

 If $TU^\star$ has only one eigenvalue $\lambda$, then $\lambda=\frac{1}{d}\left(\sum_{m=1}^N\alpha_m\right)$  since $\lambda\in \Lambda,$ and so
 $\widetilde{FP}(\{f_m\}_{m=1}^N,\{g_m\}_{m=1}^N)=\sum_{n=1}^d\lambda_n^2=\frac{1}{d}\left(\sum_{m=1}^N\alpha_m\right)^2.$
\end{proof}

\begin{remark}
Note that the bounds in (1) and (2) of Proposition~\ref{cotasfp} are not necessarily achieved, but are attained when $TU^*$ has only one eigenvalue.
\end{remark}

Our next step is to study critical pairs of sequences of our mixed frame potential.

\section{Critical pairs of sequences of the mixed frame potential}

We show now that if the mixed frame operator is the identity operator times a constant, then the sequence $\{\alpha_m\} _{m=1,...,N}$ satisfies an equality.
\begin{proposition}\label{cotas}
Let $(\{f_m\}_{m=1}^N,\{g_m\}_{m=1}^N) \in \tilde{S}(\{\alpha_m\}_{m=1}^N$ be such that $TU^\star =A \Id $ with $A\in K.$ Then $\frac{1}{d}\sum_{i=1}^N\alpha_i=A.$
\end{proposition}
\begin{proof}
Let $\{e_n\}_{n=1}^d$ be an orthonormal basis in $\mathbb{H}_d.$ Since $A \Id =TU^\star$ we have that
\begin{equation*}
\begin{aligned}
\frac{1}{d}\sum_{m=1}^N\alpha_m&=\frac{1}{d}\sum_{m=1}^N\langle f_m,g_m\rangle =\frac{1}{d}\sum_{m=1}^N\langle \sum_{j=1}^d\langle f_m, e_j\rangle e_j, g_m\rangle \\
&=\frac{1}{d}\sum_{m=1}^N\sum_{j=1}^d\langle f_m, e_j\rangle \langle e_j, g_m\rangle =\frac{1}{d}\sum_{j=1}^d A\langle e_j,e_j\rangle =A
\end{aligned}
\end{equation*}
\end{proof}

\bigskip
In order to state the following results we will need some definitions.

Let $\mathcal{L}$ be a finite index set.
\begin{definition}
 We call $\{f_m\}_{m\in \mathcal{L}}\subset \mathbb{H}_d$ and $\{g_m\}_{m\in \mathcal{L}}\subset \mathbb{H}_d$
 {\em generalized biorthogonal sequences} if there exists $\{\alpha_m\}_{m\in \mathcal{L}}\subset K_{\neq 0}$ such that
\begin{equation}
\left\{
  \begin{array}{ll}
   \langle f_n,g_m\rangle =0, & \hbox{for all}~~ n\neq m; \\
    \langle f_m,g_m\rangle =\alpha_m, & \hbox{for all}~~ m\in \mathcal{L}.
  \end{array}
\right.
\end{equation}
\end{definition}

\bigskip

\begin{definition}
Let $A\in K.$ We say $\{f_m\}_{m\in \mathcal{L}}\subset \mathbb{H}_d$  and $\{g_m\}_{m\in \mathcal{L}}\subset \mathbb{H}_d$ are $A$-{\em generalized dual frames} if
\begin{equation}
  \left\{
    \begin{array}{ll}
      \displaystyle\sum_{m\in \mathcal{L}}\langle f,g_m\rangle  f_m=A f , & \hbox{for all}~~f\in span\{f_m\}_{m\in \mathcal{L}}~~\hbox{and}  \\
      \displaystyle\sum_{m\in \mathcal{L}}\langle f,f_m\rangle  g_m=\overline{A} f, & \hbox{for all}~~f\in span\{g_m\}_{m\in \mathcal{L}}.
    \end{array}
  \right.
\end{equation}
\end{definition}

In the following we will see that the critical points of the real or the imaginary part of the restricted mixed frame potential satisfy certain Lagrange equations.
\begin{proposition}\label{lexs}
Let $\{\alpha_n\}_{m=1}^N\subset K.$ If
$(\{f_m\}_{m=1}^N,\{g_m\}_{m=1}^N)$ is a local extrema or a saddle
point of the real or the imaginary part of the mixed frame potential
$\widetilde{FP}:\tilde{S}(\{\alpha_m\}_{m=1}^N)\longrightarrow K,$
then for each $m=1,...,N$ there exists $c\in K$ such that
\begin{equation}\label{cri}
\sum_{n=1,n\neq m}^N \langle f_m,g_n\rangle f_n= c f_m \text{ and }
\sum_{n=1,n\neq m}^N  \langle g_m,f_n\rangle g_n=\overline{c}g_m
\end{equation}
\end{proposition}
\begin{proof}
Consider the $m$-th  mixed frame potential denoted by $\widetilde{FP}_m,$ where
  $$\widetilde{FP}_m(f,g)= \langle f_m,g_m\rangle ^2+\sum_{n\neq m}\langle f_n,g\rangle \langle f,g_n\rangle + \widetilde{FP}(\{f_n\}_{n\neq m},\{g_n\}_{n\neq m}).$$
  Since $(\{f_m\}_{m=1}^N,\{g_m\}_{m=1}^N)$ is a local extrema or a saddle point of the real or the imaginary part of the frame potential
  $\widetilde{FP}$ restricted to $\tilde{S}(\{\alpha_m\}_{m=1}^N),$ we have that $(f_m,g_m)$ is a local extrema or a saddle point of the real or the imaginary part of $\widetilde{FP}_m$ in $S(\alpha_m)=\{(f,g)\in \Hisp \times\Hisp~~:~~ \langle f,g\rangle =\alpha_m \},$ where
  $$\widetilde{FP}_m(f,g)= \alpha_m^2+\sum_{n\neq m}\langle f_n,g\rangle \langle f,g_n\rangle + \sum_{n=1, n\neq m}^N\sum_{r=1, r\neq m}^N\langle f_n,g_r\rangle \langle g_r, f_n\rangle .$$
Hence, the corresponding several variable constrained problem must
be solved. Using Lagrange multipliers, it can be seen that there
exist $c_1,c_2\in\R$ such that
$$(I)\,\,\nabla Re(\widetilde{FP}_m) (f,g)|_{(f_m,g_m)}= c_1\nabla Re(\langle f,g\rangle )|_{(f_m,g_m)}+c_2\nabla Im(\langle f,g\rangle )|_{(f_m,g_m)}$$
or there exist $c_2,c_3\in R$ such that
$$(II)\,\nabla Im(\widetilde{FP}_m) (f,g)|_{(f_m,g_m)}= c_3\nabla Re(\langle f,g\rangle )|_{(f_m,g_m)}+c_4\nabla Im(\langle f,g\rangle )|_{(f_m,g_m)}.$$

\bigskip
From $(I)$ we have the following equations
\begin{enumerate}[(i)]
  \item $\nabla_{Re(f)}Re(\widetilde{FP}_m) (f,g)|_{(f_m,g_m)}= c_1\nabla_{Re(f)}Re(\langle f,g\rangle )|_{(f_m,g_m)}+c_2\nabla_{Re(f)}Im(\langle f,g\rangle )|_{(f_m,g_m)},$
  \item $\nabla_{Im(f)}Re(\widetilde{FP}_m) (f,g)|_{(f_m,g_m)}= c_1\nabla_{Im(f)}Re(\langle f,g\rangle )|_{(f_m,g_m)}+c_2\nabla_{Im(f)}Im(\langle f,g\rangle )|_{(f_m,g_m)},$
  
  \item $\nabla_{Re(g)}Re(\widetilde{FP}_m) (f,g)|_{(f_m,g_m)}= c_1\nabla_{Re(g)}Re(\langle f,g\rangle )|_{(f_m,g_m)}+c_2\nabla_{Re(g)}Im(\langle f,g\rangle )|_{(f_m,g_m)},$
  
  \item $\nabla_{Im(g)}Re(\widetilde{FP}_m) (f,g)|_{(f_m,g_m)}= c_1\nabla_{Im(g)}Re(\langle f,g\rangle )|_{(f_m,g_m)}+c_2\nabla_{Im(g)}Im(\langle f,g\rangle )|_{(f_m,g_m)},$
  
  \end{enumerate}
Hence, from (i) and (ii)
$$Re\left(\sum_{n=1,n\neq m}^N  \langle g_m,f_n\rangle g_n\right) =c_1Re(g_m)-c_2Im(g_m),$$
$$Im\left(\sum_{n=1,n\neq m}^N  \langle g_m,f_n\rangle g_n\right)=c_1Im(g_m)+c_2Re(g_m)$$
and from (iii) and (iv)
$$Re\left(\sum_{n=1,n\neq m}^N  \langle f_m,g_n\rangle f_n\right)=c_1 Re(f_m)+c_2Im(f_m),$$
$$Im\left(\sum_{n=1,n\neq m}^N  \langle f_m,g_n\rangle f_n\right)=c_1Im(f_m)-c_2Re(f_m),$$
which yields,
$$\sum_{n=1,n\neq m}^N  \langle g_m,f_n\rangle g_n=c_1 g_m+ic_2g_m=(c_1+ic_2)g_m$$
and
$$\sum_{n=1,n\neq m}^N  \langle f_m,g_n\rangle f_n=c_1 f_m-ic_2f_m=(c_1-ic_2)f_m,$$ so we obtain the desired result if we take $c=c_1+ic_2.$

\medskip

Observe that in a similar way we can obtain from $(II)$ that
$$\sum_{n=1,n\neq m}^N  \langle g_m,f_n\rangle g_n=(c_4-ic_3)g_m$$
and
$$\sum_{n=1,n\neq m}^N  \langle f_m,g_n\rangle f_n=(c_4+ic_3)f_m,$$
which implies in particular that if
$(\{f_m\}_{m=1}^N,\{g_m\}_{m=1}^N)$ is a local extrema or a saddle
point of the real and the imaginary part of the restricted mixed
frame potential, then $c_4=c_1$ and $c_3=-c_2.$
\end{proof}

\begin{definition}
Let $\{\alpha_m\}_{m=1}^N\subset K.$ We say that $(\{f_n\}_{m=1}^N,\{g_n\}_{m=1}^N)\in \tilde{S}(\{\alpha_m\}_{m=1}^N)$  is a {\em critical pair of sequences} if for each $m=1,...,N$ there exists $c\in K$ such that (\ref{cri}) is satisfied.
\end{definition}

Now we are ready to provide a structure of these critical pairs of sequences:

\begin{theorem}\label{critautovec}
Let $\{\alpha_m\}_{m=1}^N\subset K.$  If $(\{f_m\}_{m=1}^N,\{g_m\}_{m=1}^N)$ is a {\it critical pair of sequences}, then
\begin{enumerate}
\item for each $m\in\{1,...,N\},\,f_m$ is an eigenvector of $TU^\star$ and $g_m$ is an eigenvector of $UT^\star,$ and the corresponding eigenvalues are conjugates.
\item for $\{\lambda_j\}_{j=1}^J$ the sequence of distinct eigenvalues of $TU^\star,$ there exists a sequence of indexing sets $\{I_j \}_{j=1}^J$ with $\bigcup_{j=1}^J I_j =\{1,...,N\},$ such that $\{f_m\}_{m\in I_j}$ and $\{g_m\}_{m\in I_j}$ are $\lambda_j$-generalized dual frames.
\end{enumerate}
\end{theorem}
\begin{proof}

\begin{enumerate}
\item Since
$(\{f_n\}_{m=1}^N,\{g_n\}_{m=1}^N)\in \tilde{S}(\{\alpha_m\}_{m=1}^N)$  is a critical pair of sequences, for $m\in {1,...,N}$ there exists $c \in K$ such that
\begin{equation}
\sum_{n=1,n\neq m}^N \langle f_m,g_n\rangle f_n= c f_m \text{ and }
\sum_{n=1,n\neq m}^N  \langle g_m,f_n\rangle g_n=\overline{c}g_m.
\end{equation}
So,
$$TU^\star f_m=\langle f_m, g_m\rangle f_m+\sum_{n=1,n\neq m}^N \langle f_m,g_n\rangle f_n=\alpha_m f_m +cf_m=(\alpha_m+c)f_m,$$
and
$$UT^*g_m=\sum_{n=1,n\neq m}^N \langle g_m ,f_n\rangle g_n=\langle g_m, f_m\rangle g_m+\sum_{n=1,n\neq m}^N \langle g_m ,f_n\rangle g_n=(\overline{\alpha_m}+\overline{c})g_m,$$
i.e. $f_m$ is an eigenvector of $TU^*$ and $g_m$ is an eigenvector of $UT^*$ and the eigenvalues are conjugates.

\medskip

\item

Let $\{\lambda_j\}_{j=1}^J$ be the sequence of distinct eigenvalues of $TU^\star.$
Since $(TU^\star)^\star=UT^\star,$ the eigenvalues of $UT^\star$ are the conjugates of the eigenvalues of $TU^\star.$ We call $\{R_j\}_{j=1}^J$  the set of all right
eigenvectors of $TU^\star,$ and $\{L_j\}_{j=1}^J$ the set of all left eigenvectors of $TU^\star,$ i.e. for each $j=1,...,J$ we have:
$$R_j=\{f\in \mathbb{H}_d ~~:~~TU^\star f=\lambda_j f\}=\{f\in \mathbb{H}_d~~:~~f^\star UT^\star=\overline{\lambda_j}f^\star\}$$
$$L_j=\{g\in \mathbb{H}_d ~~:~~g^\star TU^\star =\lambda_j g^\star\}=\{g\in \mathbb{H}_d~~:~~ UT^\star g=\overline{\lambda_j} g\}$$
We know that if $i\neq j$ then $R_i \perp L_j.$

Let  $\{I_j\}_{j=1}^J$ be the sequence of indexing sets given by $$I_j=\{m\in\{1,...,N\}:~~TU^\star f_m = \lambda_j f_m ~~\hbox{and}~~UT^\star g_m=\overline{\lambda_j} g_m\}.$$
Take $j\in\{1,...,J\}$ and $f\in R_j.$ If $m\notin I_j$ then $m\in I_i$ for some $i\neq j,$ hence $g_m\in L_i$ following that $\langle f,g_m\rangle =0.$ This yields
 $$\sum_{m\in I_j}\langle f,g_m\rangle  f_m= TU^*f=\lambda_j f.$$
 Analogously we obtain that for $f\in L_j$
 $$\sum_{m\in I_j}\langle f,f_m\rangle  g_m= UT^*f=\overline{\lambda_j} f,$$
 So, since $span\{f_m\}_{m\in I_j}\subseteq R_j,$ and $span\{g_m\}_{m\in I_j}\subseteq L_j,$ we have that
 \begin{equation}
  \left\{
    \begin{array}{ll}
      \displaystyle\sum_{m\in I_j}\langle f,g_m\rangle  f_m=TU^*f=\lambda_j f , & \hbox{for all}~~f\in span\{f_m\}_{m\in I_j}~~\hbox{and}  \\
      \displaystyle\sum_{m\in I_j}\langle f,f_m\rangle  g_m=UT^*f=\overline{\lambda_j} f, & \hbox{for all}~~f\in span\{g_m\}_{m\in I_j},
    \end{array}
  \right.
\end{equation}
i.e. $\{f_m\}_{m\in I_j}$ and $\{g_m\}_{m\in I_j}$ are $\lambda_j$-generalized dual frames.
Moreover, we proved that if $\lambda_j\neq 0$ then $span\{f_m\}_{m\in I_j}=R_j$ and $span\{g_m\}_{m\in I_j}=L_j.$

\end{enumerate}
\end{proof}


Now we describe the structure of the pairs that are local extrema of
the real and the imaginary part of the restricted frame potential. As we will see in Proposition~\ref{or}, under certain conditions the same structure is also valid for pairs that are local extrema of the real or the imaginary part of the restricted frame potential.
\begin{theorem}\label{labedett}
Let $\{\alpha_n\}_{n=1}^N\subset K_{\neq 0}.$ Then every pair $(\{f_m\}_{m=1}^N,\{g_m\}_{m=1}^N)$ which is a local
extrema of the real and the imaginary part of the mixed frame
potential
$\widetilde{FP}:\tilde{S}(\{\alpha_m\}_{m=1}^N)\longrightarrow K,$
can be decomposed as
$$\left(\{f_m\}_{m\in \mathcal{I}^c} \cup\{f_m\}_{m\in \mathcal{I}}~~,~~\{g_m\}_{m\in \mathcal{I}^c} \cup\{g_m\}_{m\in \mathcal{I}}\right),$$

where

\begin{enumerate}
\item[(a)] $\mathcal{I}\subseteq \{1,...,N\}$
\item[(b)] $\{f_m\}_{m\in \mathcal{I}^c}$ and $\{g_m\}_{m\in \mathcal{I}^c}$ are generalized biorthogonal sequences
\item[(c)] $\{f_m\}_{m\in \mathcal{I}}\subset \left(span\{g_m\}_{m\in \mathcal{I}^c}\right)^\perp$ and $\{g_m\}_{m\in \mathcal{I}}\subset \left({span\{f_m\}_{m\in \mathcal{I}^c}}\right)^\perp$ and $\{f_m\}_{m\in \mathcal{I}}$ and $\{g_m\}_{m\in \mathcal{I}}$  are $A$-generalized dual frames, where
    $$A=\frac{\sum_{m\in \mathcal{I}}\alpha_m}{dim\left(span\{f_m\}_{m\in \mathcal{I}}\right)}.$$
\end{enumerate}
\end{theorem}

\begin{proof}
Let $(\{f_m\}_{m=1}^N,\{g_m\}_{m=1}^N)$ be a local extrema of the real and the imaginary part of the mixed frame potential $\widetilde{FP}:\tilde{S}(\{\alpha_m\}_{m=1}^N)\longrightarrow K.$

\begin{enumerate}


\item We have that in particular $(\{f_m\}_{m=1}^N,\{g_m\}_{m=1}^N)$ is a critical pair of sequences, so by Theorem~\ref{critautovec}, for each $m\in\{1,...,N\},\,f_m$ is an eigenvector of $TU^\star$ and $g_m$ is an eigenvector of $UT^\star,$ and the corresponding eigenvalues are conjugates.

\medskip

\item Let $\{\lambda_j\}_{j=1}^J$ the sequence of distinct eigenvalues of $TU^*,$  where $\lambda_J$ is an eigenvalue of $TU^*$ which satisfies that
$|\lambda_J|\leq |\lambda_j|,$ for all $j<J.$

Take  $\{I_j\}_{j=1}^J$ the sequence of indexing sets given by $$I_j=\{m\in\{1,...,N\}:~~TU^\star f_m = \lambda_j f_m ~~\hbox{and}~~UT^\star g_m=\overline{\lambda_j} g_m\}.$$

By Theorem~\ref{critautovec}  $\{f_m\}_{m\in I_j}$ and $\{g_m\}_{m\in I_j}$ are $\lambda_j$-generalized dual frames for all $j=1,...,J.$

\medskip

\item We will show that $\{f_m\}_{m\in I_j}$ is linearly independent in $R_j$  for any $j<J.$ The proof that $\{g_m\}_{m\in
I_j}$ is linearly independent in $L_j$ for any $j<J$ is analogous.

Assume that $\{f_m\}_{m\in I_j}$ is not l.i. in $R_j$ for some $j=1,...,J-1.$
Then there exists a nonzero sequence of $\{r_m\}_{m\in I_j}\subset K$ such that $|r_m|\le \frac{1}{2}$ for all $m\in I_j$ and $\sum_{m\in I_j}\overline{r_m }\alpha_m f_m=0.$

We will assume without loss of generality that $(\{f_m\}_{m=1}^N,\{g_m\}_{m=1}^N)$ minimizes the real part of the mixed frame potential. The other cases can be proved in a similar way.
\begin{enumerate}[a)]
\item
If $Re(\lambda_J)<0$ we take $h_1\in R_J$ and $h_2\in L_J,$ such that $\langle h_1,h_2\rangle =1.$ Let
$m\in {1,...,N}$ and $u_m\in K$ such that $u_m^2=\alpha_m.$ We define for each $m\in {1,...,N}~~\Psi_m:(-1,1)\rightarrow S(\alpha_m),~~~~\Psi_m(t)=(\beta_m(t),\gamma_m(t))$ where
\begin{equation*}
  \beta_m(t)=\left\{
    \begin{array}{ll}
      \sqrt{1-sgn(Re(\alpha_m \lambda_j)) t^2|r_m|^2} f_m+tr_m u_m h_1, & m\in I_j; \\
      f_m,~~~~~~ & m\notin I_j
    \end{array}
  \right.
\end{equation*}
and
\begin{equation*}
  \gamma_m(t)=\left\{
    \begin{array}{ll}
      \sqrt{1-sgn(Re(\alpha_m \lambda_j))t^2|r_m|^2} g_m+tr_m u_m h_2, & m\in I_j; \\
      g_m,~~~~~~ & m\notin I_j.
    \end{array}
  \right.
\end{equation*}
We have that $\{\Psi_m(0)\}_{m=1}^N=(\{f_m\}_{m=1}^N,\{g_m\}_{m=1}^N)$ and $$Re(\widetilde{FP})(\{\Psi_m(t)\}_{m=1}^N)=Re \sum_{m=1}^N\sum_{n=1}^N\langle \beta_m(t), \gamma_n(t)\rangle \langle \beta_n(t),\gamma_m(t)\rangle.$$ By the product rule
\begin{equation*}
  \begin{aligned}
  \frac{d Re(\widetilde{FP})}{dt}(\{\Psi_m(0)\}_{m=1}^N)&=Re\sum_{m\in I_j}\sum_{n=1}^N\langle r_m u_m h_1,g_n\rangle \langle f_n, g_m\rangle +\\
  &+Re \sum_{m\in I_j}\sum_{n=1}^N\langle f_m,r_n u_n h_2\rangle \langle g_m,f_n\rangle +\\
  &+Re \sum_{m\in I_j}\sum_{n=1}^N\langle f_m,g_n\rangle \langle f_n, r_m u_m h_2 \rangle +\\
  &+Re \sum_{m\in I_j}\sum_{n=1}^N\langle f_m,g_n\rangle \langle r_n u_n h_1, g_m\rangle =\\
  & S_1+S_2+S_3+S_4.
\end{aligned}
\end{equation*}
$S_1=0$ since for $m\in I_j,$  we have that $\langle f_n,g_m\rangle =0$ for $n\notin I_j$ and  $\langle h_1,g_n\rangle =0$ for $n\in I_j$ because $h_1\in R_J.$ In a similar way we see that $S_2, S_3$ and $ S_4=0.$
Hence we obtain $$\frac{d Re(\widetilde{FP})}{dt}(\{\Psi_m(0)\}_{m=1}^N)=0.$$
\begin{equation*}
  \begin{aligned}
&\frac{d^2 Re(\widetilde{FP})}{dt^2}(\{\Psi_m(0)\}_{m=1}^N)=Re\sum_{m\in I_j} \sum_{n=1}^N\langle \beta_m^{''}(0),\gamma_n(0)\rangle \langle \beta_n(0),\gamma_m(0)\rangle ~+\\
&+~Re\sum_{m\in I_j}\sum_{n=1}^N \langle \beta_m^{'}(0),\gamma_n^{'}(0)\rangle \langle \beta_n(0),\gamma_m(0)\rangle +Re\sum_{m\in I_j}\sum_{n=1}^N \langle \beta_m^{'}(0),\gamma_n(0)\rangle \langle \beta_n(0),\gamma_m^{'}(0)\rangle ~+\\
&+~Re\sum_{m\in I_j}\sum_{n=1}^N \langle \beta_m^{'}(0),\gamma_n(0)\rangle \langle \beta_n^{'}(0),\gamma_m(0)\rangle +Re\sum_{m\in I_j}\sum_{n=1}^N \langle \beta_m^{'}(0),\gamma_n^{'}(0)\rangle \langle \beta_n(0),\gamma_m(0)\rangle ~+\\
&+~Re\sum_{m\in I_j}\sum_{n=1}^N \langle \beta_m(0),\gamma_n^{''}(0)\rangle \langle \beta_n(0), \gamma_m(0)\rangle +Re\sum_{m\in I_j}\sum_{n=1}^N \langle \beta_m(0),\gamma_n^{'}(0)\rangle \langle \beta_n(0),\gamma_m^{'}(0)\rangle ~+\\
&+~Re\sum_{m\in I_j}\sum_{n=1}^N \langle \beta_m(0),\gamma_n^{'}(0)\rangle \langle \beta_n^{'}(0),\gamma_m(0)\rangle +Re\sum_{m\in I_j}\sum_{n=1}^N \langle \beta_m^{'}(0),\gamma_n(0)\rangle \langle \beta_n(0),\gamma_m^{'}(0)\rangle ~+\\
&+~Re\sum_{m\in I_j}\sum_{n=1}^N \langle \beta_m(0),\gamma_n^{'}(0)\rangle \langle \beta_n(0), \gamma_m^{'}(0)\rangle +Re\sum_{m\in I_j}\sum_{n=1}^N \langle \beta_m(0),\gamma_n(0)\rangle \langle \beta_n(0),\gamma_m^{''}(0)\rangle ~+\\
&+~Re\sum_{m\in I_j}\sum_{n=1}^N \langle \beta_m(0),\gamma_n(0)\rangle \langle \beta_n^{'}(0),\gamma_m^{'}(0)\rangle +Re\sum_{m\in I_j}\sum_{n=1}^N \langle \beta_m^{'}(0),\gamma_n(0)\rangle \langle \beta_n^{'}(0),\gamma_m(0)\rangle ~+\\
&+~Re\sum_{m\in I_j}\sum_{n=1}^N \langle \beta_m(0),\gamma_n^{'}(0)\rangle \langle \beta_n^{'}(0),\gamma_m(0)\rangle +Re\sum_{m\in I_j}\sum_{n=1}^N \langle \beta_m(0),\gamma_n(0)\rangle \langle \beta_n^{'}(0), \gamma_m^{'}(0)\rangle ~+\\
&+~Re\sum_{m\in I_j}\sum_{n=1}^N \langle \beta_m(0),\gamma_n(0)\rangle \langle \beta_n^{''}(0), \gamma_m(0)\rangle =\sum_{i=1}^{16} S_i.
\end{aligned}
\end{equation*}

We obtain that $S_2=S_4=S_5=S_7=S_{10}=S_{12}=S_{13}=S_{15}=0$ and $S_1=S_6=S_{11}=S_{16}=-\sum_{m\in I_j}|r_m|^2
|Re(\alpha_m)\lambda_j|.$ For the rest of the sums
$S_3=S_8=S_9=S_{14}=Re(\lambda_J)\sum_{m\in I_j}|r_m|^2|\alpha_m| .$
Finally
$$\frac{d^2 Re (\widetilde{FP})}{dt^2}(\{\Psi_m(0)\}_{m=1}^N)=4\left(-\sum_{m\in I_j}|r_m|^2 |Re(\alpha_m \lambda_j)|+
Re(\lambda_J)\sum_{m\in I_j}|r_m|^2|\alpha_m| \right),$$ thus
$\frac{d^2 Re (\widetilde{FP})}{dt^2}(\{\Psi_m(0)\}_{m=1}^N)<0,$
since the sequence $\{r_m\}_{m=1}^N$ is nonzero by assumption.

So in $t=0$ there is a maximum of $Re(\widetilde{FP})$ restricted to $\{\Psi_m(t)\}_{m=1}^N,$ i.e. we have that  for all $t\in (-1,1)$
$$Re(\widetilde{FP})(\{\Psi_m(t)\}_{m=1}^N)< Re(\widetilde{FP})(\{\Psi_m(0)\}_{m=1}^N)=Re(\widetilde{FP})((\{f_m\}_{m=1}^N,\{g_m\}_{m=1}^N))$$
which is a  contradiction since we assumed that $(\{f_m\}_{m=1}^N,\{g_m\}_{m=1}^N)$ is a local minimizer of $Re(\widetilde{FP}).$

\item If $Re(\lambda_J)\ge 0$ we use the same function $\Psi_m(t)=(\beta_m(t),\gamma_m(t))$
as defined in a), but choose $h_1\in R_J,~h_2\in L_J$ such
that $\langle h_1,h_2\rangle =-1.$ Analogously as in a) we obtain
$$\frac{d Re(\widetilde{FP})}{dt}(\{\Psi_m(0)\}_{m=1}^N)=0.$$
For the second derivative we have
$$\frac{d^2 Re(\widetilde{FP})}{dt^2}(\{\Psi_m(0)\}_{m=1}^N)=4\left( \sum_{m\in I_j}-|r_m|^2|Re(\alpha_m\lambda_j)|-Re(\lambda_J)\sum_{m\in I_j}|r_m|^2|\alpha_m|
 \right).$$

\medskip

If $Re(\lambda_J)> 0,$ we know that $\frac{d^2 Re(\widetilde{FP})}{dt^2}(\{\Psi_m(0)\}_{m=1}^N)<0.$

So in $t=0$ there is also a maximum of $Re(\widetilde{FP})$
restricted to $\{\Psi_m(t)\}_{m=1}^N,$ which is a  again a
contradiction since $(\{f_m\}_{m=1}^N,\{g_m\}_{m=1}^N)$ is a local
minimizer of $Re(\widetilde{FP}).$

\medskip

Now consider $Re(\lambda_J)=0.$ Let  $m_0\in I_j$ such that
$r_{m_0}\neq 0.$  If $Re(\alpha_{m_0} \lambda_j)\neq 0,$ then
$\frac{d^2 Re(\widetilde{FP})}{dt^2}(\{\Psi_m(0)\}_{m=1}^N)<0$ and
we are done.

If $Re(\alpha_{m_0} \lambda_j)= 0,$ we are in the only case where we use the hypothesis that we also have a local extrema
in the imaginary part of the restricted mixed frame potential.

Observe that if  $j<J$ then $\lambda_j\neq 0,$ since
$|\lambda_j|\geq |\lambda_J|$ and $\lambda_j\neq \lambda_J.$ Also,
we have that $\alpha_{m_0}\neq 0.$  So $\alpha_{m_0} \lambda_j\neq
0,$ thus we know that if $Re(\alpha_{m_0} \lambda_j)=0,$ necessarily
$Im(\alpha_{m_0}) \lambda_j\neq 0.$ Using the same curve
$\{\Psi_m(t)\}_{m=1}^N,$ but replacing in the definition
$Re(\alpha_{m} \lambda_j)$ by $Im(\alpha_{m} \lambda_j)$ in case
there is a minimum in the imaginary part, and by $-Im(\alpha_{m}
\lambda_j)$ in case there is a maximum, we also arrive to a
contradiction for this particular case.

Hence we can conclude that $\{f_m\}_{m\in I_j}$ is linearly independent in $R_j.$

\end{enumerate}

\medskip

\item

As we observed before, if  $j<J$ then $\lambda_j\neq 0.$ Let $w_j$ be such that $w_j^2=\lambda_j.$ We will prove that $\{\frac{1}{w_j}f_n\}_{n\in I_j}$ and $\{\frac{1}{\overline{w_j}}g_n\}_{n\in I_j}$ are biorthogonal sequences for $j<J:$\\
By item (3) we have that $\{\frac{1}{w_j}f_n\}_{n\in I_j}$ is l.i. in $R_j$
and $\{\frac{1}{\overline{w_j}}g_n\}_{n\in I_j}$ is l.i in $L_j.$ In item (2) we showed that for all $f\in R_j,$
$$\sum_{m\in I_j}\langle f,g_m\rangle f_m=\lambda_j f,$$ so
$$\sum_{m\in I_j}\langle f,\frac{g_m}{\overline{w_j}}\rangle \frac{f_m}{w_j}=f.$$
We also proved that for $f\in L_j,$
$$\sum_{m\in I_j}\langle f,f_m\rangle g_m=\lambda_j f,$$
so $\{\frac{1}{w_j}f_n\}_{n\in I_j}$ and $\{\frac{1}{\overline{w_j}}g_n\}_{n\in I_j}$  are a basis of $R_j$ and $L_j$ respectively. Hence for $l\in I_j$ we have $f_l\in R_j$ and $g_l\in L_j$ and
\begin{equation*}
\begin{aligned}
  0&=\sum_{m\in I_j}\langle f_l,\frac{g_m}{\overline{w_j}}\rangle \frac{f_m}{w_j}-f_l\\
   &=\left(\langle \frac{f_l}{w_j},\frac{g_l}{\overline{w_j}}\rangle -1\right)f_l+ \sum_{m\in I_j,m\neq l}\langle \frac{f_l}{w_j},\frac{g_m}{\overline{w_j}}\rangle f_m.
\end{aligned}
\end{equation*}
Since $\{\frac{1}{w_j}f_n\}_{n\in I_j}$ is a basis in $R_j$ it follows that $\langle \frac{f_l}{w_j},\frac{g_m}{\overline{w_j}}\rangle =0$ for any $l \in I_j,$ $l\neq m,~~m, $ and $\langle \frac{f_l}{w_j},\frac{g_l}{\overline{w_j}}\rangle =1$ and so we obtain the result.

\medskip

Observe that in particular we saw that if $m\in I_j,\,j<J$ we have that $\alpha_m=\lambda_j.$

\medskip

\item
By item (2) we have that for all $f\in L_J$
$$\sum_{j\in I_J}\langle f,f_j\rangle g_j=\lambda_J f.$$
Let $\{e_n\}_{n=1}^{dim L_J}$ be an orthonormal basis in $L_J.$ Then
\begin{equation*}
  \begin{aligned}
  \frac{1}{dim L_J}\sum_{m\in I_J}\alpha_m&=\frac{1}{dim L_J}\sum_{m\in I_J} \langle f_m,g_m\rangle \\
   &=\frac{1}{dim L_J}\sum_{m\in I_J}\langle f_m,\sum_{j=1}^{dim L_J}\langle g_m, e_j\rangle e_j\rangle \\
   &=\frac{1}{dim L_J}\sum_{m\in I_J}\sum_{j=1}^{dim L_J}\langle f_m,e_j\rangle \langle g_m, e_j\rangle \\
   &=\sum_{j=1}^{dim L_J}\frac{1}{dim L_J}\langle \sum_{m\in I_J}\langle f_m,e_j\rangle g_m, e_j\rangle \\
   &=\sum_{j=1}^{dim L_J}\frac{1}{dim L_J}\lambda_J\langle e_j,e_j\rangle =\lambda_J.
\end{aligned}
\end{equation*}
Similarly, we obtain
$$\frac{1}{dim R_J}\sum_{m\in I_J}\alpha_m=\lambda_J.$$

\medskip

Finally, we obtain the decomposition

$$\{f_m\}_{m=1}^N=\{f_m\}_{m\in I_{J^c}}\cup \{f_m\}_{m\in I_J}^N$$
and
$$\{g_m\}_{m=1}^N=\{g_m\}_{m\in I_{J^c}}\cup \{g_m\}_{m\in I_J}.$$

By item (4) we have that $\{f_m\}_{m\in I_{J^c}}$ and $\{g_m\}_{m\in I_{J^c}}$ are generalized biorthogonal sequences. From item (2) and (5) it follows that $\{f_m\}_{m\in I_J}$ and $\{g_m\}_{m\in I_J}$ are $\lambda_J$-generalized dual frames where $\lambda_J=\frac{1}{dim L_J}\sum_{m\in I_J}\alpha_m.$ So, setting $I=I_J,$ we have the desired result.
\end{enumerate}
\end{proof}

As mentioned before, under some additional hypothesis we can assure the same structure for a pair that is a local extrema of the real or the imaginary part of the restricted frame potential:

\begin{proposition}\label{or}
 Let $\{\alpha_n\}_{n=1}^N\subset K_{\neq 0}$ and $(\{f_m\}_{m=1}^N,\{g_m\}_{m=1}^N)\in \tilde{S}(\{\alpha_m\}_{m=1}^N)$ such that  $TU^*$ is injective.
 \begin{enumerate}
\item If there exists an eigenvalue $\lambda_J$ of $TU^*$  such that $Re(\lambda_J)\neq 0,$ the decomposition of Theorem~\ref{labedett} can be obtained assuming only that $(\{f_m\}_{m=1}^N,\{g_m\}_{m=1}^N)$ is a local extrema of the real part of $\widetilde{FP}:\tilde{S}(\{\alpha_m\}_{m=1}^N)\longrightarrow K.$
\item If there exists an eigenvalue $\lambda_J$ of $TU^*$ such that $Im(\lambda_J)\neq 0$, the decomposition of Theorem~\ref{labedett} can be obtained assuming only that $(\{f_m\}_{m=1}^N,\{g_m\}_{m=1}^N)$ is a local extrema of the imaginary part of $\widetilde{FP}:\tilde{S}(\{\alpha_m\}_{m=1}^N)\longrightarrow K.$
\end{enumerate}
\end{proposition}

\begin{proof}
In each case the proof is the same as the proof of
Theorem~\ref{labedett}, except that we set $$I=I_J=\{m~~:~~TU^\star
f_m = \lambda_J f_m ~~\hbox{and}~~UT^\star g_m=\overline{\lambda_J}
g_m\}$$ associated to $\lambda_J$
(which now not necessarily satisfies $|\lambda_J|\leq |\lambda_j|$
for all $j<J).$ The result follows from the observations in item (3)
of the proof of Theorem~\ref{labedett}.
\end{proof}

We finally obtain the following result concerning dual frames with prescribed scalar porducts.

\begin{corollary}
Let $\{\alpha_m\}_{m=1}^N\subset K.$ Then the following statements are equivalent:
\begin{enumerate}
\item There exists $(\{f_m\}_{m=1}^N,\{g_m\}_{m=1}^N)\in \tilde{S}(\{\alpha_m\}_{m=1}^N)$ which is a pair of dual frames.
\item There exists $(\{f_m\}_{m=1}^N,\{g_m\}_{m=1}^N)$  in $\tilde{S}(\{\alpha_m\}_{m=1}^N)$ such that $TU^*$ has only real eigenvalues,  $\widetilde{FP}(\{f_m\}_{m=1}^N,\{g_m\}_{m=1}^N)=d$ and $Re(\sum_{m=1}^N \alpha_m)\geq d.$

\end{enumerate}

\end{corollary}

\begin{proof}

$(1)\Rightarrow (2)$

 Assume $(\{f_m\}_{m=1}^N,\{g_m\}_{m=1}^N)\in \tilde{S}(\{\alpha_m\}_{m=1}^N)$ is a pair of dual frames. Then $TU^*= \Id$ and so $1$ is the only eigenvalue of $TU^*,$ which implies that $\widetilde{FP}(\{f_m\}_{m=1}^N,\{g_m\}_{m=1}^N)=d.$ By Proposition~\ref{cotas} we have $\sum_{m=1}^N \alpha_m= d,$ hence in particular $Re(\sum_{m=1}^N \alpha_m)\geq d.$

\medskip

$(2)\Rightarrow (1)$

Take $(\{f_m\}_{m=1}^N,\{g_m\}_{m=1}^N)\in \tilde{S}(\{\alpha_m\}_{m=1}^N)$ such that $TU^*$ has only real eigenvalues,  $\widetilde{FP}(\{f_m\}_{m=1}^N,\{g_m\}_{m=1}^N)=d$ and $Re(\sum_{m=1}^N \alpha_m)\geq d.$

Since all the eigenvalues of $TU^*$ are real, by Proposition~\ref{cotasfp} we have that

$\widetilde{FP}(\{f_m\}_{m=1}^N,\{g_m\}_{m=1}^N)$ and $\sum_{m=1}^N \alpha_m$ are real and $\widetilde{FP}(\{f_m\}_{m=1}^N,\{g_m\}_{m=1}^N)=d \ge \frac{1}{d} \left(\sum_{m=1}^N\alpha_m\right)^2.$ Since $\sum_{m=1}^N\alpha_m=Re(\sum_{m=1}^N \alpha_m)\geq d,$  we obtain $d=\frac{1}{d} \left(\sum_{m=1}^N\alpha_m\right)^2,$ and so $(\{f_m\}_{m=1}^N,\{g_m\}_{m=1}^N)$ attains the lower bound of the restricted frame potential. Hence, as we could see in the proof of Proposition~\ref{cotasfp},  $TU^*$ has only one eigenvalue equal to $\frac{1}{d} \sum_{m=1}^N\alpha_m=1$

On the other hand, $(\{f_m\}_{m=1}^N,\{g_m\}_{m=1}^N)$ is then also a local minima. So, by item (5) of the proof of Theorem~\ref{labedett}, $\frac{1}{d} \sum_{m=1}^N\alpha_m=\frac{1}{dim L_J} \sum_{m=1}^N\alpha_m,$  which says that
$dim L_J=d,$ i.e. $(\{f_m\}_{m=1}^N,\{g_m\}_{m=1}^N)$ is a dual frame.
\end{proof}

\begin{remark}
If we assume $N>d,$ the statements in the previous corollary are also equivalent to say that $\sum_{m=1}^N \alpha_m= d.$ This is a consequence of Proposition~\ref{cotas} and of Corollary 3.7 in \cite{CPX12}.
\end{remark}


\vspace{13pt}

\centerline{ACKNOWLEDGMENT}

The authors want to thank Ole Christensen for useful discussions concerning this paper.

S. Heineken acknowledges the support of the Intra-European Marie
Curie Fellowship (FP7 project PIEF-GA-2008-221090), UBACyT 2011-2014
(UBA) and PICT 2011-0436 (ANPCyT).

The research of I. Carrizo was supported by the EUCETIFA project of the University of Vienna, CONICET, Universidad Nacional de San Luis and the Technical University of Denmark.


\bibliographystyle{amsalpha}


\end{document}